\documentclass{article}

\usepackage{amsfonts,amsmath, amsthm, amssymb}
\usepackage[latin1]{inputenc}
\usepackage[english]{babel}
\usepackage{graphicx}
\usepackage{enumitem}
\usepackage{calc}

\newcommand{\NN}{\mathbb{N}}
\newcommand{\RR}{\mathbb{R}}

\newcommand{\D}{\mathcal{D}}

\newcommand{\HH}{\mathcal{H}}

\newcommand{\sd}{\rm sd}
\newcommand{\dist}{\rm d}

\newcommand{\len}{\rm len}
\newcommand{\mult}{\rm m}
\newcommand{\odd}{\rm odd}
\newcommand{\even}{\rm even}

\newcommand{\sub}{\subseteq}

\newcommand{\comment}[1]{}

\newtheorem{theorem}{Theorem}[section]
\newtheorem{lemma}[theorem]{Lemma}
\newtheorem{corollary}[theorem]{Corollary}
\newtheorem{proposition}[theorem]{Proposition}

\theoremstyle{definition}
\newtheorem*{question}{Problem}
\newtheorem*{conjecture}{Conjecture}

\def\?#1{\vadjust{\vbox to 0pt{\vss\vskip-8pt\leftline{%
     \llap{\hbox{\vbox{\pretolerance=-1
     \doublehyphendemerits=0\finalhyphendemerits=0
     \hsize16truemm\tolerance=10000\small
     \lineskip=0pt\lineskiplimit=0pt
     \rightskip=0pt plus16truemm\baselineskip8pt\noindent
     \hskip0pt        
     #1\endgraf}\hskip7truemm}}}\vss}}}
     \date{}

     \title{Steiner trees and higher geodecity}
     \author{Daniel Wei\ss{}auer}
     
\begin{document}
		\maketitle
		
		\begin{abstract}
		Let~$G$ be a connected graph and $\ell : E(G) \to \RR^+$ a length-function on the edges of~$G$. The \emph{Steiner distance} $\sd_G(A)$ of $A \sub V(G)$ within~$G$ is the minimum length of a connected subgraph of~$G$ containing~$A$, where the length of a subgraph is the sum of the lengths of its edges. 
		
		It is clear that every subgraph $H \sub G$, with the induced length-function $\ell|_{E(H)}$, satisfies $\sd_H(A) \geq \sd_G(A)$ for every $A \sub V(H)$. We call $H \sub G$ \emph{$k$-geodesic in~$G$} if equality is attained for every $A \sub V(H)$ with $|A| \leq k$. A subgraph is \emph{fully geodesic} if it is $k$-geodesic for every $k \in \NN$. It is easy to construct examples of graphs $H \sub G$ such that~$H$ is $k$-geodesic, but not $(k+1)$-geodesic, so this defines a strict hierarchy of properties. We are interested in situations in which this hierarchy collapses in the sense that if $H \sub G$ is $k$-geodesic, then~$H$ is already fully geodesic in~$G$.
		
		Our first result of this kind asserts that if~$T$ is a tree and $T \sub G$ is 2-geodesic with respect to some length-function~$\ell$, then it is fully geodesic. This fails for graphs containing a cycle. We then prove that if~$C$ is a cycle and $C \sub G$ is 6-geodesic, then~$C$ is fully geodesic. We present an example showing that the number~6 is indeed optimal.
		
		We then develop a structural approach towards a more general theory and present several open questions concerning the big picture underlying this phenomenon.
	\end{abstract}     
     
     \begin{section}{Introduction}
     	\label{intro}
     	
     	Let~$G$ be a graph and $\ell : E(G) \to \RR^+$ a function that assigns to every edge $e \in E(G)$ a positive \emph{length} $\ell(e)$. This naturally extends to subgraphs $H \sub G$ as $\ell(H) := \sum_{e \in E(H)} \ell(e)$. The \emph{Steiner distance} $\sd_G(A)$ of a set $A \sub V(G)$ is defined as the minimum length of a connected subgraph of~$G$ containing~$A$, where $\sd_G(A) := \infty$ if no such subgraph exists. Every such minimizer is necessarily a tree and we say it is a \emph{Steiner tree for~$A$ in~$G$}. In the case where $A = \{ x, y\}$, the Steiner distance of~$A$ is the ordinary distance~$\dist_G(x,y)$ between~$x$ and~$y$. Hence this definition yields a natural extension of the notion of ``distance'' for sets of more than two vertices. Corresponding notions of radius, diameter and convexity have been studied in the literature \cite{chartrand, steinerdiamdeg, steinerdiamgirth, steinerdiamplanar, steinerconvexnhood, steinerconvexgeom}. Here, we initiate the study of \emph{Steiner geodecity}, with a focus on structural assumptions that cause a collapse in the naturally arising hierarchy.

     	Let $H \sub G$ be a subgraph of~$G$, equipped with the length-function $\ell|_{E(H)}$. It is clear that for every $A \sub V(H)$ we have $\sd_H(A) \geq \sd_G(A)$. For a natural number~$k$, we say that~$H$ is \emph{$k$-geodesic in~$G$} if $\sd_H(A) = \sd_G(A)$ for every $A \sub V(H)$ with $|A| \leq k$. We call~$H$ \emph{fully geodesic in~$G$} if it is $k$-geodesic for every $k \in \mathbb{N}$.
     	
     	By definition, a $k$-geodesic subgraph is $m$-geodesic for every $m \leq k$. In general, this hierarchy is strict: In Section~\ref{general theory} we provide, for every $k \in \NN$, examples of graphs $H \sub G$ and a length-function $\ell : E(G) \to \RR^+$ such that~$H$ is $k$-geodesic, but not $(k+1)$-geodesic. On the other hand, it is easy to see that if $H \sub G$ is a 2-geodesic \emph{path}, then it is necessarily fully geodesic, because the Steiner distance of any $A \sub V(H)$ in~$H$ is equal to the maximum distance between two $a, b \in A$. Our first result extends this to all trees.
     	
     	\begin{theorem} \label{tree 2-geo}
			Let~$G$ be a graph with length-function~$\ell$ and $T \sub G$ a tree. If~$T$ is 2-geodesic in~$G$, then it is fully geodesic.     	     	
     	\end{theorem}
     	
     Here, it really is necessary for the subgraph to be acyclic (see Corollary~\ref{H_2 forests}). Hence the natural follow-up question is what happens in the case where the subgraph is a cycle.
     	
     	\begin{theorem} \label{cycle 6-geo}
     		Let~$G$ be a graph with length-function~$\ell$ and $C \sub G$ a cycle. If~$C$ is 6-geodesic in~$G$, then it is fully geodesic.
     	\end{theorem}
     	
     	Note that the number~6 cannot be replaced by any smaller integer. 
     	
	In Section~\ref{preliminaries} we introduce notation and terminology needed in the rest of the paper. Section~\ref{toolbox} contains observations and lemmas that will be used later. We then prove Theorem~\ref{tree 2-geo} in Section~\ref{sct on trees}. In Section~\ref{sct on cycles} we prove Theorem~\ref{cycle 6-geo} and provide an example showing that the number~6 is optimal. Section~\ref{general theory} contains an approach towards a general theory, aiming at a deeper understanding of the phenomenon displayed in Theorem~\ref{tree 2-geo} and Theorem~\ref{cycle 6-geo}. Finally, we take the opportunity to present the short and easy proof that in any graph~$G$ with length-function~$\ell$, the cycle space of~$G$ is generated by the set of fully geodesic cycles.
     	
	\end{section}     	
	
	\begin{section}{Preliminaries}
	\label{preliminaries}
		
		All graphs considered here are finite and undirected. It is convenient for us to allow parallel edges. In particular, a cycle may consist of just two vertices joined by two parallel edges. Loops are redundant for our purposes and we exclude them to avoid trivialities. Most of our notation and terminology follows that of~\cite{diestelbook}, unless stated otherwise.
		
		A set~$A$ of vertices in a graph~$G$ is called \emph{connected} if and only if $G[A]$ is.
		
		Let $G, H$ be two graphs. A \emph{model of~$G$ in~$H$} is a family of disjoint connected \emph{branch-sets} $B_v \sub V(H)$, $v \in V(G)$, together with an injective map $\beta : E(G) \to E(H)$, where we require that for any $e \in E(G)$ with endpoints $u, v \in V(G)$, the edge $\beta(e) \in E(H)$ joins vertices from~$B_u$ and~$B_v$. We say that~$G$ \emph{is a minor of~$H$} if~$H$ contains a model of~$G$.
	
		We use additive notation for adding or deleting vertices and edges. Specifically, let~$G$ be a graph, $H$ a subgraph of~$G$, $v \in V(G)$ and $e =xy \in E(G)$. Then $H + v$ is the graph with vertex-set $V(H) \cup \{ v \}$ and edge-set $E(H) \cup \{ vw \in E(G) \colon w \in V(H) \}$. Similarly, $H + e$ is the graph with vertex-set $V(H) \cup \{ x, y\}$ and edge-set $E(H) \cup \{ e \}$.
%
		
		Let~$G$ be a graph with length-function~$\ell$. A \emph{walk} in~$G$ is an alternating sequence $W = v_1e_1v_2 \ldots e_kv_{k+1}$ of vertices~$v_i$ and edges~$e_i$ such that $e_i = v_i v_{i+1}$ for every $1 \leq i \leq k$. The walk~$W$ is \emph{closed} if $v_1 = v_{k+1}$. Stretching our terminology slightly, we define the \emph{length} of the walk as $\len_G(W) := \sum_{1 \leq i \leq k} \ell(e_i)$. The \emph{multiplicity}~$\mult_W(e)$ of an edge $e \in E(G)$ is the number of times it is traversed by~$W$, that is, the number of indices $1 \leq j \leq k$ with $e = e_j$. It is clear that
		\begin{equation} \label{length walk}
			\len_G(W) = \sum_{e \in E(G)} \mult_W(e) \ell(e) .
		\end{equation}
		
		Let~$G$ be a graph and~$C$ a cycle with $V(C) \sub V(G)$. We say that a walk~$W$ in~$G$ is \emph{traced by~$C$} in~$G$ if it can be obtained from~$C$ by choosing a starting vertex $x \in V(C)$ and an orientation~$\overrightarrow{C}$ of~$C$ and replacing every $\overrightarrow{ab} \in E(\overrightarrow{C})$ by a shortest path from~$a$ to~$b$ in~$G$. A cycle may trace several walks, but they all have the same length: Every walk~$W$ traced by~$C$ satisfies
		\begin{equation} \label{length traced walk}			
			\len_G(W) = \sum_{ab \in E(C)} \dist_G(a, b) .		
		\end{equation}
		 Even more can be said if the graph~$G$ is a tree. Then all the shortest $a$-$b$-paths for $ab \in E(C)$ are unique and all walks traced by~$C$ differ only in their starting vertex and/or orientation. In particular, every walk~$W$ traced by~$C$ in a tree~$T$ satisfies
		\begin{equation}		\label{multiplicities traced walk tree}
		\forall e \in E(T): \, \,	\mult_W(e) = | \{ ab \in E(C) \colon e \in aTb \} | ,
		\end{equation}
	where $aTb$ denotes the unique $a$-$b$-path in~$T$.		
		
%
%
		
		Let~$T$ be a tree and $X \sub V(T)$. Let $e \in E(T)$ and let $T_1^e, T_2^e$ be the two components of $T -e$. In this manner, $e$ induces a bipartition $X = X_1^e \cup X_2^e$ of~$X$, given by $X_i^e = V(T_i^e) \cap X$ for $i \in \{ 1, 2 \}$. We say that the bipartition is \emph{non-trivial} if neither of $X_1^e, X_2^e$ is empty. The set of leaves of~$T$ is denoted by~$L(T)$. If $L(T) \sub X$, then every bipartition of~$X$ induced by an edge of~$T$ is non-trivial.
		
		Let~$G$ be a graph with length-function~$\ell$, $A \sub V(G)$ and~$T$ a Steiner tree for~$A$ in~$G$. Since $\ell(e) >0 $ for every $e \in E(G)$, every leaf~$x$ of~$T$ must lie in~$A$, for otherwise $T - x$ would be a tree of smaller length containing~$A$.
		
		In general, Steiner trees need not be unique. If~$G$ is a tree, however, then every $A \sub V(G)$ has a unique Steiner tree given by $\bigcup_{a, b \in A} aTb$.
	\end{section}
	
	\begin{section}{The toolbox}
	\label{toolbox}
		The first step in all our proofs is a simple lemma that guarantees the existence of a particularly well-behaved substructure that witnesses the failure of a subgraph to be $k$-geodesic.
		
		
		Let~$H$ be a graph, $T$ a tree and~$\ell$ a length-function on~$T \cup H$. We call~$T$ a \emph{shortcut tree for~$H$} if the following hold:
			\begin{enumerate}[ itemindent=0.8cm, label=(SCT\,\arabic*)]
				\item $V(T) \cap V(H) = L(T)$, \label{sct vert}
				\item $E(T) \cap E(H) = \emptyset$, \label{sct edge}
				\item $\ell(T) < \sd_H( L(T) )$, \label{sct shorter}
				\item For every $B \subsetneq L(T)$ we have $\sd_H(B) \leq \sd_T(B)$. \label{sct minim}
			\end{enumerate}
		Note that, by definition, $H$ is not $|L(T)|$-geodesic in~$T \cup H$.
%
				
					\begin{lemma} \label{shortcut tree}
				Let~$G$ be a graph with length-function~$\ell$, $k$ a natural number and $H \sub G$. If~$H$ is not $k$-geodesic in~$G$, then~$G$ contains a shortcut tree for~$H$ with at most~$k$ leaves.
			\end{lemma}
			
%
			
			\begin{proof}
	Among all $A \sub V(H)$ with $|A| \leq k$ and $\sd_G(A) < \sd_H(A)$, choose~$A$ such that $\sd_G(A)$ is minimum. Let $T \sub G$ be a Steiner tree for~$A$ in~$G$. We claim that~$T$ is a shortcut tree for~$H$.
	
	\textit{Claim 1:} $L(T) = A = V(T) \cap V(H)$.
	
	The inclusions $L(T) \sub A \sub V(T) \cap V(H)$ are clear. We show $V(T) \cap V(H) \sub L(T)$. Assume for a contradiction that $x \in V(T) \cap V(H)$ had degree $d \geq 2$ in~$T$. Let $T_1, \ldots, T_d$ be the components of $T - x$ and for $j \in [d]$ let $A_j := A \cap V(T_j) \cup \{ x \}$. Since $L(T) \sub A$, every tree~$T_i$ contains some $a \in A$ and so $A \not \sub A_j$. In particular $|A_j| \leq k$. Moreover $\sd_G(A_j) \leq \ell( T_j + x) < \ell(T)$, so by our choice of~$A$ and~$T$ it follows that $\sd_G(A_j) = \sd_H(A_j)$. Therefore, for every $j \in [d]$ there exists a connected $S_j \sub H$ with $A_j \sub V(S_j)$ and $\ell(S_j) \leq \ell(T_j + x)$. But then $S := \bigcup_j S_j \sub H$ is connected, contains~$A$ and satisfies
	\[
	\ell(S) \leq \sum_{j = 1}^d \ell(S_j) \leq \sum_{j=1}^d \ell(T_j + x) = \ell(T) ,
	\]		
	which contradicts the fact that $\sd_H(A) > \ell(T)$ by choice of~$A$ and~$T$.	
	
	\textit{Claim 2:} $E(T) \cap E(H) = \emptyset$.
	
	Assume for a contradiction that $xy \in E(T) \cap E(H)$. By Claim~1, $x, y \in L(T)$ and so~$T$ consists only of the edge~$xy$. But then $T \sub H$ and $\sd_H(A) \leq \ell(T)$, contrary to our choice of~$A$ and~$T$.
	
	\textit{Claim 3:} $\ell(T) < \sd_H(L(T))$.
	
	We have $\ell(T) = \sd_G(A) < \sd_H(A)$. By Claim~1, $A = L(T)$.
	
	\textit{Claim 4:} For every $B \subsetneq L(T)$ we have $\sd_H(B) \leq \sd_T(B)$.
	
		Let $B \subsetneq L(T)$ and let $T' := T - (A \setminus B)$. By Claim~1, $T'$ is the tree obtained from~$T$ by chopping off all leaves not in~$B$ and so 
		\[
		\sd_G(B) \leq \ell(T') < \ell(T) = \sd_G(A) .
		\]
		By minimality of~$A$, it follows that $\sd_H(B) = \sd_G(B) \leq \sd_T(B)$.
			\end{proof}
			
			Our proofs of Theorem~\ref{tree 2-geo} and Theorem~\ref{cycle 6-geo} proceed by contradiction and follow a similar outline. Let $H \sub G$ be a subgraph satisfying a certain set of assumptions. The aim is to show that~$H$ is fully geodesic. Assume for a contradiction that it was not and apply Lemma~\ref{shortcut tree} to find a shortcut tree~$T$ for~$H$. Let~$C$ be a cycle with $V(C) \sub L(T)$ and let $W_H, W_T$ be walks traced by~$C$ in~$H$ and~$T$, respectively. If $|L(T)| \geq 3$, then it follows from~(\ref{length traced walk}) and~\ref{sct minim} that $\len(W_H) \leq \len(W_T)$. 
			
			Ensure that $\mult_{W_T}(e) \leq 2$ for every $e \in E(T)$ and that $\mult_{W_H}(e) \geq 2$ for all $e \in E(S)$, where $S \sub H$ is connected with $L(T) \sub V(S)$. Then
			\[
			2\, \sd_H(L(T)) \leq 2\, \ell(S) \leq \len(W_H) \leq \len(W_T) \leq 2 \, \ell(T) ,
			\]
	which contradicts~\ref{sct shorter}.		

	The first task is thus to determine, given a tree~$T$, for which cycles~$C$ with $V(C) \sub V(T)$ we have $m_W(e) \leq 2$ for all $e \in E(T)$, where~$W$ is a walk traced by~$C$ in~$T$. Let $S \sub T$ be the Steiner tree for~$V(C)$ in~$T$. It is clear that~$W$ does not traverse any edges $e \in E(T) \setminus E(S)$ and $L(S) \sub V(C) \sub V(S)$. Hence we can always reduce to this case and may for now assume that $S = T$ and $L(T) \sub V(C)$.

	\begin{lemma} \label{pos even}
		Let~$T$ be a tree, $C$ a cycle with $L(T) \sub V(C) \sub V(T)$ and~$W$ a walk traced by~$C$ in~$T$. Then $m_W(e)$ is positive and even for every $e \in E(T)$.
	\end{lemma}
	
	\begin{proof}
		Let $e \in E(T)$ and let $V(C) = V(C)_1 \cup V(C)_2$ be the induced bipartition. Since $L(T) \sub V(C)$, this bipartition is non-trivial. By~(\ref{multiplicities traced walk tree}), $m_W(e)$ is the number of $ab \in E(C)$ such that $e \in aTb$. By definition, $e \in aTb$ if and only if~$a$ and~$b$ lie in different sides of the bipartition. Every cycle has a positive even number of edges across any non-trivial bipartition of its vertex-set.
	\end{proof}

		\begin{lemma} \label{equality achieved}
		Let~$T$ be a tree, $C$ a cycle with $L(T) \sub V(C) \sub V(T)$. Then
		\[
		2 \ell(T) \leq \sum_{ab \in E(C)} \dist_T(a, b) .
		\]
		Moreover, there is a cycle~$C$ with $V(C) = L(T)$ for which equality holds.
		\end{lemma}
		
		\begin{proof}
			Let~$W$ be a walk traced by~$C$ in~$T$. By Lemma~\ref{pos even}, (\ref{length walk}) and~(\ref{length traced walk})
			\[
			2 \ell(T) \leq \sum_{e \in E(T)} \mult_W(e) \ell(e) = \len(W) = \sum_{ab \in E(C)} \dist_T(a, b) .
			\]
	To see that equality can be attained, let~$2T$ be the multigraph obtained from~$T$ by doubling all edges. Since all degrees in~$2T$ are even, it has a Eulerian trail~$W$, which may be considered as a walk in~$T$ with $\mult_W(e) = 2$ for all $e \in E(T)$. This walk traverses the leaves of~$T$ in some cyclic order, which yields a cycle~$C$ with $V(C) = L(T)$. It is easily verified that~$W$ is traced by~$C$ in~$T$ and so
	\[
	2 \ell(T) = \sum_{e \in E(T)} \mult_W(e) \ell(e) = \len(W) = \sum_{ab \in E(C)} \dist_T(a, b) .
	\]
		\end{proof}
		
		We have now covered everything needed in the proof of Theorem~\ref{tree 2-geo}, so the curious reader may skip ahead to Section~\ref{sct on trees}.

	\begin{figure}
		\begin{center}
			\includegraphics[width=2.5cm]{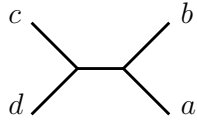}
			\label{4leaftree}
			\caption{A tree with four leaves}
		\end{center}
	\end{figure}		
		
		In general, not every cycle~$C$ with $V(C) = L(T)$ achieves equality in Lemma~\ref{equality achieved}. Consider the tree~$T$ from Figure~\ref{4leaftree} and the following three cycles on $L(T)$
		\[
		C_1 = abcda, \, \, C_2 = acdba, \, \, C_3 = acbda . 
		\]
		
		\begin{figure}
			\begin{center}
				\includegraphics[width=6cm]{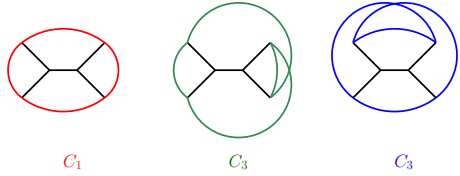}
				\label{4leaftree3cyc}
				\caption{The three cycles on~$T$}
			\end{center}
		\end{figure}
		
		For the first two, equality holds, but not for the third one. But how does~$C_3$ differ from the other two? It is easy to see that we can add~$C_1$ to the planar drawing of~$T$ depicted in Figure~\ref{4leaftree}: There exists a planar drawing of $T \cup C_1$ extending this particular drawing. This is not true for~$C_2$, but it can be salvaged by exchanging the positions of~$a$ and~$b$ in Figure~\ref{4leaftree}. Of course, this is merely tantamount to saying that $T \cup C_i$ is planar for $i \in \{ 1, 2\}$. 
		
%

	On the other hand, it is easy to see that $T \cup C_3$ is isomorphic to~$K_{3,3}$ and therefore non-planar.

		\begin{lemma} \label{euler planar}
		Let~$T$ be a tree and~$C$ a cycle with $V(C) = L(T)$. Let~$W$ be a walk traced by~$C$ in~$T$. The following are equivalent:
		\begin{enumerate}[label=(a)]
			\item $T \cup C$ is planar.
			\item For every $e \in E(T)$, both $V(C)_1^e, V(C)_2^e$ are connected in~$C$.
			\item $W$ traverses every edge of~$T$ precisely twice.
		\end{enumerate}
	\end{lemma}
	
	\begin{proof}
		(a) $\Rightarrow$ (b): Fix a planar drawing of $T \cup C$. The closed curve representing~$C$ divides the plane into two regions and the drawing of~$T$ lies in the closure of one of them. By symmetry, we may assume that it lies within the closed disk inscribed by~$C$. Let $A \sub V(C)$ disconnected and choose $a, b \in A$ from distinct components of~$C[A]$. $C$ is the disjoint union of two edge-disjoint $a$-$b$-paths $S_1, S_2$ and both of them must meet $C \setminus A$, say $c \in V(S_1) \setminus A$ and $d \in V(S_2) \setminus A$.
		
		The curves representing $aTb$ and $cTd$ lie entirely within the disk and so they must cross. Since the drawing is planar, $aTb$ and~$cTd$ have a common vertex. In particular, $A$ cannot be the set of leaves within a component of $T - e$ for any edge $e \in E(T)$.
		
		(b) $\Rightarrow$ (c): Let $e \in E(T)$. By assumption, there are precisely two edges $f_1, f_2 \in E(C)$ between $V(C)_1^e$ and $V(C)_2^e$. These edges are, by definition, the ones whose endpoints are separated in~$T$ by~$e$. By~(\ref{multiplicities traced walk tree}), $m_W(e) =2$.
		
		(c) $\Rightarrow$ (a): For $ab \in E(C)$, let $D_{ab} := aTb + ab \sub T \cup C$. The set $\D := \{ D_{ab} \colon ab \in E(C) \}$ of all these cycles is the fundamental cycle basis of $T \cup C$ with respect to the spanning tree~$T$. Every edge of~$C$ occurs in only one cycle of~$\D$. By assumption and~(\ref{multiplicities traced walk tree}), every edge of~$T$ lies on precisely two cycles in~$\D$. Covering every edge of the graph at most twice, the set~$\D$ is a \emph{sparse basis} of the cycle space of $T \cup C$. By MacLane's Theorem, $T \cup C$ is planar.
	\end{proof}
			
	\end{section}

	\begin{section}{Shortcut trees for trees}
	\label{sct on trees}
	
%
%
	
		\begin{proof}[Proof of Theorem~\ref{tree 2-geo}]
		
			Assume for a contradiction that $T \sub G$ was not fully geodesic and let $R \sub T$ be a shortcut tree for~$T$. Let $T' \sub T$ be the Steiner tree for $L(R)$ in~$T$. By Lemma~\ref{equality achieved}, there is a cycle~$C$ with $V(C) = L(R)$ such that
			\[
			2 \ell(R) = \sum_{ab \in E(C)} \dist_R(a, b) .
			\]
			Note that~$T'$ is 2-geodesic in~$T$ and therefore in~$G$, so that $\dist_{T'}(a,b) \leq \dist_R(a,b)$ for all $ab \in E(C)$. Since every leaf of~$T'$ lies in $L(R) = V(C)$, we can apply Lemma~\ref{equality achieved} to~$T'$ and~$C$ and conclude
			\[
			2 \ell(T') \leq \sum_{ab \in E(C)} \dist_{T'}(a, b) \leq  \sum_{ab \in E(C)} \dist_R(a, b) = 2\ell(R) ,
			\]
			which contradicts~\ref{sct shorter}.
		\end{proof}
	
	\end{section}

	\begin{section}{Shortcut trees for cycles}
	\label{sct on cycles}
		
	By Lemma~\ref{shortcut tree}, it suffices to prove the following.		
		
		\begin{theorem} \label{shortcut tree cycle strong}
			Let~$T$ be a shortcut tree for a cycle~$C$. Then $T \cup C$ is a subdivision of one of the five (multi-)graphs in Figure~\ref{five shortcut trees}. In particular, $C$ is not 6-geodesic in $T \cup C$.
		\end{theorem}
		
		\begin{figure}[h]
	\begin{center}
	\includegraphics[width=10cm]{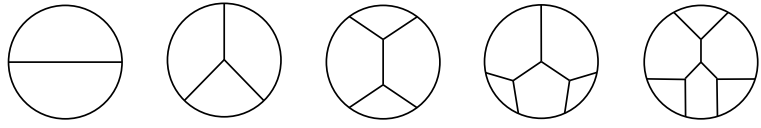}
	\caption{The five possible shortcut trees for a cycle}
	\label{five shortcut trees}
\end{center}					
		\end{figure}
		
		Theorem~\ref{shortcut tree cycle strong} is best possible in the sense that for each of the graphs in Figure~\ref{five shortcut trees} there exists a length-function which makes the tree inside a shortcut tree for the outer cycle, see Figure~\ref{cycleshortcutlength}. These length-functions were constructed in a joint effort with Pascal Gollin and Karl Heuer in an ill-fated attempt to prove that a statement like Theorem~\ref{cycle 6-geo} could not possibly be true.
		
			\begin{figure}[h]
				\begin{center}
					\includegraphics[width=10cm]{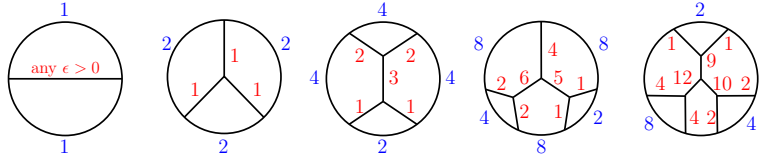}
				\end{center}
					\label{cycleshortcutlength}
					\caption{Shortcut trees for cycles}
			\end{figure}

		This section is devoted entirely to the proof of Theorem~\ref{shortcut tree cycle strong}. Let~$T$ be a shortcut tree for a cycle~$C$ with length-function $\ell : E(T \cup C) \to \RR^+$ and let $L := L(T)$. 
			     	
	The case where $|L| = 2$ is trivial, so we henceforth assume that $|L| \geq 3$. By suppressing any degree-2 vertices, we may assume without loss of generality that $V(C) = L(T)$ and that~$T$ contains no vertices of degree~2. 			     		
	
				     			\begin{lemma} \label{cover disjoint trees}
			     			Let $T_1, T_2 \sub T$ be edge-disjoint trees. For $i \in \{ 1, 2\}$, let $L_i := L \cap V(T_i)$. If $L = L_1 \cup L_2$ is a non-trivial bipartition of~$L$, then both $C[L_1], C[L_2]$ are connected.
			     			\end{lemma}
			     			
			     			\begin{proof}
			     				By~\ref{sct minim} there are connected $S_1, S_2 \sub C$ with $\ell(S_i) \leq \sd_T(L_i) \leq \ell(T_i)$ for $i \in \{ 1, 2 \}$. Assume for a contradiction that~$C[L_1]$ was not connected. Then $V(S_1) \cap L_2$ is non-empty and $S_1 \cup S_2 $ is connected, contains~$L$ and satisfies
			     	\[
			     	\ell( S_1 \cup S_2) \leq \ell(S_1) + \ell(S_2) \leq \ell(T_1) + \ell(T_2) \leq \ell(T) ,
			     	\]
			     	which contradicts~\ref{sct shorter}.
			     			\end{proof}		     	
			     	
			     	\begin{lemma} \label{planar 3reg}
			     		$T \cup C$ is planar and 3-regular.
			     	\end{lemma}
			     	
			     			\begin{proof}
			     	Let $e \in E(T)$, let $T_1, T_2$ be the two components of $T - e$ and let $L = L_1 \cup L_2$ be the induced (non-trivial) bipartition of~$L$. By Lemma~\ref{cover disjoint trees}, both $C[L_1], C[L_2]$ are connected. Therefore $T \cup C$ is planar by Lemma~\ref{euler planar}.
			     	
			     	
			     	To see that $T \cup C$ is 3-regular, it suffices to show that no $t \in T$ has degree greater than~3 in~$T$. We just showed that $T \cup C$ is planar, so fix some planar drawing of it. Suppose for a contradiction that $t \in T$ had $d \geq 4$ neighbors in~$T$. In the drawing, these are arranged in some cyclic order as $t_1, t_2, \ldots, t_d$. For $j \in [d]$, let $R_j := T_j + t$, where~$T_j$ is the component of $T - t$ containing~$t_j$. Let~$T_{\odd}$ be the union of all~$R_j$ for odd $j \in [d]$ and~$T_{\even}$ the union of all~$R_j$ for even $j \in [d]$. Then $T_{\odd}, T_{\even} \sub T$ are edge-disjoint and yield a nontrivial bipartition $L = L_{\odd} \cup L_{\even}$ of the leaves. But neither of $C[L_{\odd}], C[L_{\even}]$ is connected, contrary to Lemma~\ref{cover disjoint trees}.
			     	
%
%
			     			\end{proof}

			\begin{lemma} \label{consec cycle long}
		Let $e_0 \in E(C)$ arbitrary. Then for any two consecutive edges $e_1, e_2$ of~$C$ we have $\ell(e_1) + \ell(e_2) > \ell(e_0)$. In particular $\ell(e_0) < \ell(C)/2$.
	\end{lemma}
	
	\begin{proof}
		Suppose that $e_1, e_2 \in E(C)$ are both incident with $x \in L$. Let $S \sub C$ be a Steiner tree for $B := L \setminus \{ x \}$ in~$C$. By~\ref{sct minim} and~\ref{sct shorter} we have
		\[
		\ell(S) \leq \sd_T(B) \leq \ell(T) < \sd_C(L) .
		\]
		Thus $x \notin S$ and $E(S) = E(C) \setminus \{ e_1, e_2 \}$. Thus $P := C - e_0$ is not a Steiner tree for~$B$ and we must have $\ell(P) > \ell(S)$.
	\end{proof}

		Let $t \in T$ and~$N$ its set of neighbors in~$T$. For every $s \in N$ the set~$L_s$ of leaves~$x$ with $s \in tTx$ is connected in~$C$. Each $C[L_s]$ has two edges $f_s^1, f_s^2 \in E(C)$ incident to it.

		\begin{lemma} \label{good root}
		There is a $t \in T$ such that for every $s \in N$ and any $f \in \{ f_s^1, f_s^2 \}$ we have $\ell(C[L_s] + f) < \ell(C)/2$.
		\end{lemma}
		
		\begin{proof}
			We construct a directed graph~$D$ with $V(D) = V(T)$ as follows. For every $t \in T$, draw an arc to any $s \in N$ for which $\ell(C[L_s] + f_s^i) \geq \ell(C)/2$ for some $i \in \{ 1, 2 \}$.
			
			\textit{Claim:} If $\overrightarrow{ts} \in E(D)$, then $\overrightarrow{st} \notin E(D)$.
			
			Assume that there was an edge $st \in E(T)$ for which both $\overrightarrow{st}, \overrightarrow{ts} \in E(D)$. Let $T_s, T_t$ be the two components of $T - st$, where $s \in T_s$, and let $L = L_s \cup L_t$ be the induced bipartition of~$L$. By Lemma~\ref{cover disjoint trees}, both $C[L_s]$ and $C[L_t]$ are connected paths, say with endpoints $a_s, b_s$ and $a_t, b_t$ (possibly $a_s = b_s$ or $a_t = b_t$) so that $a_sa_t \in E(C)$ and $b_sb_t \in E(C)$ (see Figure~\ref{badneighbors}). Without loss of generality $\ell(a_sa_t) \leq \ell(b_sb_t)$. Since $\overrightarrow{ts} \in E(D)$ we have $\ell(C[L_t] + b_sb_t) \geq \ell(C)/2$ and therefore $C[L_s] + a_sa_t$ is a shortest $a_t$-$b_s$-path in~$C$. 
			Similarly, it follows from $\overrightarrow{st} \in E(D)$ that $\dist_C( a_s, b_t) = \ell(C[L_t] + a_sa_t)$.
			
	\begin{figure}
		\begin{center}
			\includegraphics[width=5cm]{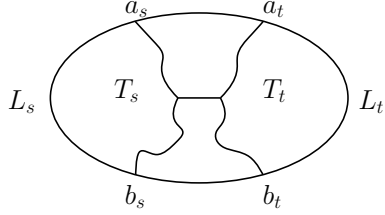}
		\end{center}
		\label{badneighbors}
		\caption{The setup in the proof of Lemma~\ref{good root}}
\end{figure}				
			
			Consider the cycle $Q := a_tb_sa_sb_ta_t$ and let $W_T, W_C$ be walks traced by~$Q$ in~$T$ and in~$C$, respectively. Then $\len(W_T) \leq 2 \, \ell(T)$, whereas
			\[
				\len(W_C) = 2 \, \ell (C - b_sb_t) \geq 2 \, \sd_C(L) .
			\]
			By~\ref{sct minim} we have $\dist_C(x,y) \leq \dist_T(x,y)$ for all $x, y \in L$ and so $\len(W_C) \leq \len(W_T)$. But then $\sd_C(L) \leq \ell(T)$, contrary to~\ref{sct shorter}. This finishes the proof of the claim.
			
			Since every edge of~$D$ is an orientation of an edge of~$T$ and no edge of~$T$ is oriented both ways, it follows that~$D$ has at most $|V(T)| - 1$ edges. Since~$D$ has $|V(T)|$ vertices, there is a $t \in V(T)$ with no outgoing edges.
		\end{proof}
		
			Fix a node $t \in T$ as guaranteed by the previous lemma. If~$t$ was a leaf with neighbor~$s$, say, then $\ell(f_s^1) = \ell(C) - \ell(C[L_s] + f_s^2) > \ell(C)/2$ and, symmetrically, $\ell(f_s^2) > \ell(C)/2$, which is impossible. Hence by Lemma~\ref{planar 3reg}, $t$ has three neighbors $s_1, s_2, s_3 \in T$ and we let $L_i := C[L_{s_i}]$ and $\ell_i := \ell(L_i)$. There are three edges $f_1, f_2, f_3 \in E(C) \setminus \bigcup E(L_i)$, where $f_1$ joins~$L_1$ and~$L_2$, $f_2$ joins~$L_2$ and~$L_3$ and~$f_3$ joins~$L_3$ and~$L_1$. Each~$L_i$ is a (possibly trivial) path whose endpoints we label $a_i, b_i$ so that, in some orientation, the cycle is given by
			\[
			C = a_1L_1b_1 + f_1 + a_2L_2b_2 + f_2 + a_3L_3b_3 + f_3 .
			\]
		Hence $f_1 = b_1a_2$, $f_2 = b_2a_3$ and $f_3 = b_3a_1$ (see Figure~\ref{trail Q}).
     	
     	The fact that $\ell_1 + \ell(f_1) \leq \ell(C)/2$ means that $L_1 + f_1$ is a shortest $a_1$-$a_2$-path in~$C$ and so $\dist_C(a_1, a_2) = \ell_1 + \ell(f_1)$. Similarly, we thus know the distance between all other pairs of vertices with just one segment~$L_i$ and one edge~$f_j$ between them. 
     	
     		\begin{figure}[htb]
	\begin{center}
		\includegraphics[width=4cm]{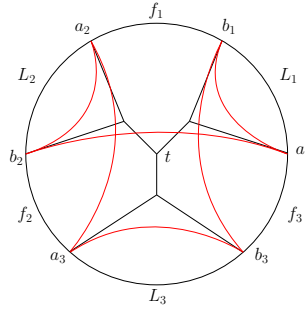}
	\end{center}
	\label{trail Q}
	\caption{The cycle~$Q$}
\end{figure}	

	If $|L_i| \leq 2$ for every $i \in [3]$, then $T \cup C$ is a subdivision of one the graphs depicted in Figure~\ref{five shortcut trees} and we are done. Hence from now on we assume that at least one~$L_i$ contains at least~3 vertices.

		\begin{lemma} \label{jumps}
			Suppose that $\max \{ |L_s| \colon s \in N \} \geq 3$. Then there is an $s \in N$ with $\ell( f_s^1 + C[L_s] + f_s^2) \leq \ell(C)/2$.
		\end{lemma}
		
		
				\begin{proof}
	
	For $j \in [3]$, let $r_j := \ell( f_{s_j}^1 + L_j + f_{s_j}^2)$. Assume wlog that $|L_1| \geq 3$. Then~$L_1$ contains at least two consecutive edges, so by Lemma~\ref{consec cycle long} we must have $\ell_1 > \ell(f_2)$. Therefore
	\[
		r_2 + r_3 = \ell(C) + \ell(f_2) - \ell_1 < \ell(C) ,
	\]
	so the minimum of $r_2, r_3$ is less than $\ell(C)/2$.
		\end{proof}
		
		By the previous lemma, we may wlog assume that 
		\begin{equation} \label{jump edge}
			\ell(f_2) + \ell_3 + \ell(f_3) \leq \ell(C)/2 ,
		\end{equation}
		 so that $f_2 + L_3 + f_3$ is a shortest $a_1$-$b_2$-path in~$C$. Together with the inequalities from Lemma~\ref{good root}, this will lead to the final contradiction.
		
		Consider the cycle $Q = a_1b_2a_2a_3b_3b_1a_1$ (see Figure~\ref{trail Q}). Let~$W_T$ be a walk traced by~$Q$ in~$T$. Every edge of~$T$ is traversed at most twice, hence
     	
  	\begin{equation}
  		\sum_{ab \in E(Q)}  \dist_T(a, b) = \ell(W_T)  \leq 2\ell(T) . \label{sum through tree}
  	\end{equation}
  	  	
  	Let~$W_C$ be a walk traced by~$Q$ in~$C$. Using~(\ref{jump edge}) and the inequalities from Lemma~\ref{good root}, we see that
  	\begin{align*}
  	\ell(W_C) &= \sum_{ab \in E(Q)}  \dist_C(a, b) = 2 \ell_1 + 2 \ell_2 + 2 \ell_3 + 2 \ell(f_2) + 2 \ell(f_3) \\
		  	&= 2 \ell(C) - 2 \ell(f_1) .
  	\end{align*}
		  	
  	But by~\ref{sct minim} we have $\dist_C(a,b) \leq \dist_T(a,b)$ for all $a, b \in L(T)$ and therefore $\ell(W_C) \leq \ell(W_T)$. Then by~(\ref{sum through tree})
  	\[
  	2 \ell(C) - 2 \ell(f_1) = \ell(W_C) \leq \ell(W_T) \leq 2 \ell(T).
  	\] 
	But then $S := C - f_1$ is a connected subgraph of~$C$ with $L(T) \sub V(S)$ satisfying $\ell(S) \leq \ell(T)$. This contradicts~\ref{sct shorter} and finishes the proof of Theorem~\ref{shortcut tree cycle strong}. \hfill \qed

	\end{section}

	\begin{section}{Towards a general theory}
\label{general theory}	
	
	We have introduced a notion of higher geodecity based on the concept of the Steiner distance of a set of vertices. This introduces a hierarchy of properties: Every $k$-geodesic subgraph is, by definition, also $m$-geodesic for any $m < k$. This hierarchy is strict in the sense that for every~$k$ there are graphs~$G$ and $H \sub G$ and a length-function~$\ell$ on~$G$ such that~$H$ is $k$-geodesic in~$G$, but not $(k+1)$-geodesic. To see this, let~$G$ be a complete graph with $V(G) = [k+1] \cup \{ 0 \}$ and let~$H$ be the subgraph induced by $[k+1]$. Define $\ell(0j) := k-1$ and $\ell(ij) := k$ for all $i, j \in [k+1]$. If~$H$ was not $k$-geodesic, then~$G$ would contain a shortcut tree~$T$ for~$H$ with $|L(T)| \leq k$. Then~$T$ must be a star with center~$0$ and 
			\[
			\ell(T) = (k-1)|L(T)| \geq k(|L(T)|-1) .
			\]
	But any spanning tree of $H[L(T)]$ has length $k(|L(T)|-1) $ and so $\sd_H(L(T)) \leq \ell(T)$, contrary to~\ref{sct shorter}. Hence~$H$ is a $k$-geodesic subgraph of~$G$. However, the star~$S$ with center~$0$ and $L(S) = [k+1]$ shows that 
	\[
	\sd_G(V(H)) \leq (k+1)(k-1) < k^2 = \sd_H(V(H)) = k^2 .
	\]
			
			Theorem~\ref{tree 2-geo} and Theorem~\ref{cycle 6-geo} demonstrate a rather strange phenomenon by providing situations in which this hierarchy collapses. 
			
			For a given natural number $k \geq 2$, let us denote by~$\HH_k$ the class of all graphs~$H$ with the property that whenever~$G$ is a graph with $H \sub G$ and~$\ell$ is a length-function on~$G$ such that~$H$ is $k$-geodesic, then~$H$ is also fully geodesic. 
			
			By definition, this yields an ascending sequence $\HH_2 \sub \HH_3 \sub \ldots $ of classes of graphs. By Theorem~\ref{tree 2-geo} all trees lie in~$\HH_2$. By Theorem~\ref{cycle 6-geo} all cycles are contained in~$\HH_6$. The example above shows that $K_{k+1} \notin \HH_k$.
			
			We now describe some general properties of the class~$\HH_k$.
			
	\begin{theorem} \label{H_k minor closed}
		For every natural number $k \geq 2$, the class~$\HH_k$ is closed under taking minors.
\end{theorem}				

	To prove this, we first provide an easier characterization of the class~$\HH_k$.
			
			\begin{proposition} \label{H_k sct}
				Let $k \geq 2$ be a natural number and~$H$ a graph. Then $H \in \HH_k$ if and only if every shortcut tree for~$H$ has at most~$k$ leaves.
			\end{proposition}
				
				\begin{proof}
					Suppose first that $H \in \HH_k$ and let~$T$ be a shortcut tree for~$H$. By~\ref{sct shorter}, $H$ is not $|L(T)|$-geodesic in $T \cup H$. Let~$m$ be the minimum integer such that~$H$ is not $m$-geodesic in $T \cup H$. By Lemma~\ref{shortcut tree}, $T \cup H$ contains a shortcut tree~$S$ with at most~$m$ leaves for~$H$. But then by~\ref{sct vert} and~\ref{sct edge}, $S$ is the Steiner tree in~$T$ of $B := L(S) \sub L(T)$. If $B \subsetneq L(T)$, then $\ell(S) = \sd_T(B) \geq \sd_H(B)$ by~\ref{sct minim}, so we must have $B = L(T)$ and $m \geq |L(T)|$. Thus~$H$ is $(|L(T)| - 1)$-geodesic in $T \cup H$, but not $|L(T)|$-geodesic. As $H \in \HH_k$, it must be that $|L(T)| - 1 < k$.
					
					Suppose now that every shortcut tree for~$H$ has at most~$k$ leaves and let $H \sub G$ $k$-geodesic with respect to some length-function $\ell : E(G) \to \RR^+$. If~$H$ was not fully geodesic, then~$G$ contained a shortcut tree~$T$ for~$H$. By assumption, $T$ has at most~$k$ leaves. But then $\sd_G( L(T)) \leq \ell(T) < \sd_H(L(T))$, so~$H$ is not $k$-geodesic in~$G$.
				\end{proof}
				
	\begin{lemma} \label{wlog connected}
		Let $k \geq 2$ be a natural number and~$G$ a graph. Then $G \in \HH_k$ if and only if every component of~$G$ is in~$\HH_k$.
	\end{lemma}			
	
	\begin{proof}
		Every shortcut tree for a component~$K$ of~$G$ becomes a shortcut tree for~$G$ by taking $\ell(e) := 1$ for all $e \in E(G) \setminus E(K)$. Hence if $G \in \HH_k$, then every component of~$G$ is in~$\HH_k$ as well. 
		
		Suppose now that every component of~$G$ is in~$\HH_k$ and that~$T$ is a shortcut tree for~$G$. If there is a component~$K$ of~$G$ with $L(T) \sub V(K)$, then~$T$ is a shortcut tree for~$K$ and so $|L(T)| \leq k$ by assumption. Otherwise, let $t_1 \in L(T) \cap V(K_1)$ and $t_2 \in L(T) \cap V(K_2)$ for distinct components $K_1, K_2$ of~$G$. By~\ref{sct minim}, it must be that $L(T) = \{ t_1, t_2 \}$ and so $|L(T)| = 2 \leq k$.
	\end{proof}

				\begin{lemma} \label{sct minor closed}
					Let~$G, H$ be two graphs and let~$T$ be a shortcut tree for~$G$. If~$G$ is a minor of~$H$, then there is a shortcut tree~$T'$ for~$H$ which is isomorphic to~$T$.
				\end{lemma}

				\begin{proof}
				Since~$G$ is a minor of~$H$, there is a family of disjoint connected sets $B_v \sub V(H)$, $v \in V(G)$, and an injective map $\beta : E(G) \to E(H)$ such that for $uv \in E(G)$, the end vertices of $\beta(uv) \in E(H)$ lie in~$B_u$ and~$B_v$.

	Let~$T$ be a shortcut tree for~$G$ with $\ell : E(T \cup G) \to \RR^+$. By adding a small positive real number to every $\ell(e)$, $e \in E(T)$, we may assume that the inequalities in~\ref{sct minim} are strict, that is
	\[
		\sd_G(B) \leq \sd_T(B) - \epsilon
	\]
	for every $B \sub L(T)$ with $2 \leq |B| < |L(T)|$, where $\epsilon > 0$ is some constant.
	
	Obtain the tree~$T'$ from~$T$ by replacing every $t \in L(T)$ by an arbitrary $x_t \in B_t$ and every $t \in V(T) \setminus L(T)$ by a new vertex~$x_t$ not contained in $V(H)$, maintaining the adjacencies. It is clear by definition that $V(T') \cap V(H) = L(T')$ and $E(T') \cap E(H) = \emptyset$. We now define a length-function $\ell' : E(T' \cup H) \to \RR^+$ as follows.
	
	For every edge $st \in E(T)$, the corresponding edge $x_sx_t \in E(T')$ receives the same length $\ell'(x_sx_t) := \ell(st)$. Every $e \in E(H)$ that is contained in one of the branchsets~$B_v$ is assigned the length $\ell '(e) := \delta$, where $\delta := \epsilon / |E(H)| $. For every $e \in E(G)$ we let $\ell '( \beta(e)) := \ell(e)$. To all other edges of~$H$ we assign the length~$\ell(T) + 1$.
	
	We now show that~$T'$ is a shortcut tree for~$H$ with the given length-function~$\ell'$. Suppose that $S' \sub H$ was a connected subgraph with $L(T') \sub V(S')$ and $\ell'(S') \leq \ell'(T')$. By our choice of~$\ell'$, every edge of~$S'$ must either lie in a branchset~$B_v$ or be the image under~$\beta$ of some edge of~$G$, since otherwise $\ell'(S') > \ell(T) = \ell'(T')$. Let $S \sub V(G)$ be the subgraph where $v \in V(S)$ if and only if $V(S') \cap B_v$ is non-empty and $e \in E(S)$ if and only if $\beta(e) \in E(S')$. Since~$S'$ is connected, so is~$S$: For any non-trivial bipartition $V(S) = U \cup W$ the graph~$S'$ contains an edge between $\bigcup_{u \in U} B_u$ and $\bigcup_{w \in W} B_w$, which in turn yields an edge of~$S$ between~$U$ and~$W$. Moreover $L(T) \sub V(S)$, since~$V(S')$ contains~$x_t$ and thus meets~$B_t$ for every $t \in L(T)$. Finally, $\ell(S) \leq \ell(S')$, which contradicts our assumption that~$T$ is a shortcut tree for~$G$.
	
	For $B' \sub L(T')$ with $2 \leq |B'| < |L(T')|$, let $B := \{ t \in T \colon x_t \in B' \}$. By assumption, there is a connected $S \sub G$ with $B \sub V(S)$ and $\ell(S) \leq \sd_T(B) - \epsilon$. Let
	\[
	S' := \bigcup_{v \in V(S)} H[B_v] + \{ \beta(e) \colon  e \in E(S) \}.
	\]
	For every $x_t \in B'$ we have $t \in B \sub V(S)$ and so $x_t \in B_t \sub V(S')$. Since~$S$ is connected and every $H[B_v]$ is connected, $S'$ is connected as well. Moreover
	\[
	\ell'(S') \leq \delta |E(H)| + \ell(S) \leq \sd_T(B) = \sd_{T'}(B') .
	\]
				\end{proof}
	
	\begin{proof}[Proof of Theorem~\ref{H_k minor closed}]
		Let~$H$ be a graph in~$\HH_k$ and~$G$ a minor of~$H$. Let~$T$ be a shortcut tree for~$G$. By Lemma~\ref{sct minor closed}, $H$ has a shortcut tree~$T'$ which is isomorphic to~$T$. By Proposition~\ref{H_k sct} and assumption on~$H$, $T$ has $|L(T')| \leq k$ leaves. Since~$T$ was arbitrary, it follows from Proposition~\ref{H_k sct} that $G \in \HH_k$.
\end{proof}		
	
	\begin{corollary} \label{H_2 forests}
		$\HH_2$ is the class of forests.
	\end{corollary}
	
	\begin{proof}
			By Theorem~\ref{tree 2-geo} and Lemma~\ref{wlog connected}, every forest is in~$\HH_2$. On the other hand, if~$G$ contains a cycle, then it contains the triangle~$C_3$ as a minor. We saw in Section~\ref{sct on cycles} that~$C_3$ has a shortcut tree with~3 leaves. By Lemma~\ref{sct minor closed}, so does~$G$ and hence $G \notin \HH_2$ by Proposition~\ref{H_k sct}.
	\end{proof}
	
	\begin{corollary} \label{num leaf bounded}
		For every natural number $k \geq 2$ there exists an integer $m = m(k)$ such that every graph that is not in~$\HH_k$ has a shortcut tree with more than~$k$, but not more than~$m$ leaves.
	\end{corollary}
	
	\begin{proof}
		Let $k \geq 2$ be a natural number. By Theorem~\ref{H_k minor closed} and the Graph Minor Theorem of Robertson and Seymour~\cite{graphminorthm} there is a finite set~$R$ of graphs such that for every graph~$H$ we have $H \in \HH_k$ if and only if~$H$ does not contain any graph in~$R$ as a minor. Let $m(k) := \max_{G \in R} |G|$.
		
		Let~$H$ be a graph and suppose $H \notin \HH_k$. Then~$H$ contains some $G \in R$ as a minor. By Proposition~\ref{H_k sct}, this graph~$G$ has a shortcut tree~$T$ with more than~$k$, but certainly at most~$|G|$ leaves. By Lemma~\ref{sct minor closed}, $H$ has a shortcut tree isomorphic to~$T$.
	\end{proof}
		
	We remark that we do not need the full strength of the Graph Minor Theorem here: We will see in a moment that the tree-width of graphs in~$\HH_k$ is bounded for every $k \geq 2$, so a simpler version of the Graph Minor Theorem can be applied, see~\cite{excludeplanar}. Still, it seems that Corollary~\ref{num leaf bounded} ought to have a more elementary proof.
		
	\begin{question}
		Give a direct proof of Corollary~\ref{num leaf bounded} that yields an explicit bound on~$m(k)$. What is the smallest possible value for~$m(k)$?
	\end{question}		
	
	In fact, we are not even aware of any example that shows one cannot simply take $m(k) = k+1$.
		
		Given that~$\HH_2$ is the class of forests, it seems tempting to think of each class~$\HH_k$ as a class of ``tree-like'' graphs. In fact, containment in~$\HH_k$ is related to the tree-width of the graph, but the relation is only one-way.

		\begin{proposition} \label{low tw example}
			For any integer $k \geq 1$, the graph $K_{2, 2k}$ is not in~$\HH_{2k-1}$.
		\end{proposition}
		
		\begin{proof}
			Let~$H$ be a complete bipartite graph $V(H) = A \cup B \cup \{ x, y \}$ with $|A| = |B| = k$, where $uv \in E(H)$ if and only if $u \in A \cup B$ and $v \in \{ x, y \}$ (or vice versa). We construct a shortcut tree for $H \cong K_{2,2k}$ with~$2k$ leaves. 
			
			For $x', y' \notin V(H)$, let~$T$ be the tree with $V(T) = A \cup B \cup \{ x', y' \}$, where~$x'$ is adjacent to every $a \in A$, $y'$ is adjacent to every $b \in B$ and $x'y' \in E(T)$. It is clear that $V(T) \cap V(H) = L(T)$ and~$T$ and~$H$ are edge-disjoint. We now define a length-function $\ell : E(T \cup H) \to \RR^+$ that turns~$T$ into a shortcut tree for~$H$.
			
			For all $a \in A$ and all $b \in B$, let
			
			\begin{gather*}
				\ell ( a x) = \ell (a x' ) = \ell ( b y) = \ell (b y') = k-1, \\
				\ell ( a y ) = \ell ( a y') = \ell ( b x ) = \ell ( b x' ) = k, \\
				\ell ( x' y' ) = k - 1 .
			\end{gather*}
			
			Let $A' \sub A, B' \sub B$. We determine $\sd_H(A' \cup B')$. By symmetry, it suffices to consider the case where $|A'| \geq |B'|$. We claim that
			\[
				\sd_H( A' \cup B') = (k-1)|A' \cup B'| + |B'| .
			\]
			It is easy to see that $\sd_H( A' \cup B') \leq (k-1)|A'| + k|B'|$, since $S^* := H[A' \cup B' \cup \{ x \}]$ is connected and achieves this length. Let now $S \sub H$ be a tree with $A' \cup B' \sub V(S)$. 
			
			If every vertex in $A' \cup B'$ is a leaf of~$S$, then~$S$ can only contain one of~$x$ and~$y$, since it does not contain a path from~$x$ to~$y$. But then~$S$ must contain one of $H[A' \cup B' \cup \{x \}]$ and $H[ A' \cup B' \cup \{ y \} ]$, so $\ell(S) \geq \ell(S^*)$.
			
			Suppose now that some $x \in A' \cup B'$ is not a leaf of~$S$. Then~$x$ has two incident edges, one of length~$k$ and one of length~$k-1$. For $s \in S$, let $r(s)$ be the sum of the lengths of all edges of~$S$ incident with~$s$. Then $\ell(s) \geq k-1$ for all $s \in A' \cup B'$ and $\ell(x) \geq 2k-1$. Since $A' \cup B'$ is independent in~$H$ (and thus in~$S$), it follows that
			\begin{align*}
			\ell(S) &\geq \sum_{s \in A' \cup B'} r(s) \geq |( A' \cup B') \setminus \{ x \}| (k-1) + (2k-1) \\
			&= |A' \cup B'|(k-1) + k \geq (k-1)|A' \cup B'| + |B'| .	
			\end{align*}
	Thus our claim is proven. For $A', B'$ as before, it is easy to see that
	\[
		\sd_T( A' \cup B') = \begin{cases}
			(k-1)|A' \cup B'| , &\text{ if } B' = \emptyset \\
			(k-1)|A' \cup B'| + k- 1, &\text{ otherwise.}
		\end{cases}
	\]			
	We thus have $\sd_T( A' \cup B' ) < \sd_H(A' \cup B')$ if and only if $|A'| = |B'| = k$. Hence~\ref{sct shorter} and~\ref{sct minim} are satisfied and~$T$ is a shortcut tree for~$H$ with~$2k$ leaves.			
			
		\end{proof}

		Note that the graph $K_{2,k}$ is planar and has tree-width~2. Hence there is no integer~$m$ such that all graphs of tree-width at most~2 are in~$\HH_m$. Using Theorem~\ref{H_k minor closed}, we can turn Proposition~\ref{low tw example} into a positive result, however.
		
		\begin{corollary} \label{H_k exclude K_2,k+2}
			For any $k \geq 2$, no $G \in \HH_k$ contains $K_{2, k+2}$ as a minor.			\hfill \qed
		\end{corollary}
	
	In particular, it follows from the Grid-Minor Theorem~\cite{excludeplanar} and planarity of~$K_{2,k}$ that the tree-width of graphs in~$\HH_k$ is bounded. Bodlaender et al~\cite{excludeK2k} gave a more precise bound for this special case, showing that graphs excluding $K_{2,k}$ as a minor have tree-width at most~$2(k-1)$.
	
%
%
		
%
		
		It seems plausible that a qualitative converse to Corollary~\ref{H_k exclude K_2,k+2} might hold.
			
			\begin{question}			\label{K_2,k as culprit}
	Is there a function $q : \NN \to \NN$ such that every graph that does not contain $K_{2,k}$ as a minor is contained in~$\HH_{q(k)}$?			
			\end{question}
%
		
	Since no subdivision of a graph~$G$ contains $K_{2, |G| + e(G) + 1}$ as a minor, a positive answer would prove the following.
		
		\begin{conjecture}
			For every graph~$G$ there exists an integer~$m$ such that every subdivision of~$G$ lies in~$\HH_m$.
		\end{conjecture}

	\end{section}

     	     	\begin{section}{Generating the cycle space}
	Let~$G$ be a graph with length-function~$\ell$. It is a well-known fact (see e.g.~\cite[Chapter~1, exercise~37]{diestelbook}) that the set of 2-geodesic cycles generates the cycle space of~$G$. This extends as follows, showing that fully geodesic cycles abound.
		
     	\begin{proposition} \label{generate cycle space}
     		Let~$G$ be a graph with length-function~$\ell$. The set of fully geodesic cycles generates the cycle space of~$G$.
     	\end{proposition}
		
		We remark, first of all, that the proof is elementary and does not rely on Theorem~\ref{cycle 6-geo}, but only requires Lemma~\ref{shortcut tree} and Lemma~\ref{pos even}.
		
		Let~$\D$ be the set of all cycles of~$G$ which cannot be written as a 2-sum of cycles of smaller length. The following is well-known.
		
		\begin{lemma}			
			The cycle space of~$G$ is generated by~$\D$.
		\end{lemma}
		
		\begin{proof}
			It suffices to show that every cycle is a 2-sum of cycles in~$\D$. Assume this was not the case and let $C \sub G$ be a cycle of minimum length that is not a 2-sum of cycles in~$\D$. In particular, $C \notin \D$ and so there are cycles $C_1, \ldots, C_k$ with $C = C_1 \oplus \ldots \oplus C_k$ and $\ell(C_i) < \ell(C)$ for every $i \in [k]$. By our choice of~$C$, every~$C_i$ can be written as a 2-sum of cycles in~$\D$. But then the same is true for~$C$, which is a contradiction.
		\end{proof}
	
	\begin{proof}[Proof of Proposition~\ref{generate cycle space}]
		We show that every $C \in \D$ is fully geodesic. Indeed, let $C \sub G$ be a cycle which is not fully geodesic and let $T\sub G$ be a shortcut tree for~$C$. There is a cycle~$D$ with $V(D) = L(T)$ such that~$C$ is a union of edge-disjoint $L(T)$-paths $P_{ab}$ joining~$a$ and~$b$ for $ab \in E(D)$.
		
		For $ab \in E(D)$ let $C_{ab} := aTb + P_{ab}$. Every edge of~$C$ lies in precisely one of these cycles. An edge $e \in E(T)$ lies in $C_{ab}$ if and only if $e \in aTb$. By Lemma~\ref{pos even} and~(\ref{multiplicities traced walk tree}), every $e \in E(T)$ lies in an even number of cycles~$C_{ab}$. Therefore $C = \bigoplus_{ab \in E(D)} C_{ab}$.
		
		For every $ab \in E(D)$, $C$ contains a path~$S$ with $E(S) = E(C) \setminus E(P_{ab})$ with $L(T) \sub V(S)$. Since~$T$ is a shortcut tree for~$C$, it follows from~\ref{sct shorter} that
		\[
		\ell(C_{ab}) \leq \ell(T) + \ell(P_{ab}) < \ell(S) + \ell(P_{ab}) = \ell(C) .
		\]
		In particular, $C \notin \D$.
	\end{proof}
	
	The fact that 2-geodesic cycles generate the cycle space has been extended to the topological cycle space of locally finite graphs graphs by Georgakopoulos and Spr\"{u}ssel~\cite{agelos}. Does Proposition~\ref{generate cycle space} have a similar extension?

	\end{section}

 		\newpage 
	\bibliographystyle{plain}
\bibliography{steinerbib}

%
%
%

     \end{document}